\documentclass[]{amsart}   
\usepackage{amssymb, accents}
\usepackage{tikz, float, tabu, enumitem, mathtools}
\usetikzlibrary{calc,decorations.pathreplacing,matrix,arrows,positioning}
\usepackage[breaklinks]{hyperref}

\newcommand{\ZZ}{\mathbb{Z}}

\newcommand{\V}{\mathcal{V}}
\newcommand{\A}{\mathbb{A}}
\newcommand{\m}[1]{\mathbb{#1}}    
\newcommand{\cl}[1]{\mathcal{#1}}  

\theoremstyle{plain} \newtheorem{thm}{Theorem}[section]

\newtheorem{cor}[thm]{Corollary}

\theoremstyle{definition} 
\theoremstyle{remark} 
\theoremstyle{plain} \newtheorem*{claim}{Claim}

\newenvironment{claimproof} {
    \begin{proof}[Proof of claim]
    
  } {
    \end{proof}
  }

\DeclareMathOperator{\Sg}{\text{Sg}}
\DeclareMathOperator{\Con}{\text{Con}}

\newcommand{\pmat}[1]{\ensuremath{ \begin{pmatrix} #1 \end{pmatrix} }}

\numberwithin{equation}{section}  
\renewcommand{\phi}{\varphi}
\renewcommand{\epsilon}{\varepsilon}

\newcommand{\sts}[1]{\undertilde{\mathbf{#1}} }  
\newcommand{\1}{\mathbf{1}}
\newcommand{\5}{\mathbf{5}}
\renewcommand{\S}{\mathbf{S}}
\renewcommand{\P}{\mathbf{P}}
\newcommand{\Hom}{\text{Hom}}

\begin{document}
\title[Dualizable algebras omitting $\{ \1, \5 \}$ have a cube term.]
  {Naturally dualizable algebras omitting types $\1$ and $\5$ have a cube
  term}
\author{Matthew Moore}
\date{\today}

\email{matthew.moore@vanderbilt.edu}
\address{
  Vanderbilt University;
  Nashville, TN 37240;
  U.S.A.}

\begin{abstract}
An early result in the theory of Natural Dualities is that an algebra with a
near unanimity (NU) term is dualizable. A converse to this is also true: if
$\V(\A)$ is congruence distributive and $\A$ is dualizable, then $\A$ has an NU
term. An important generalization of the NU term for congruence distributive
varieties is the cube term for congruence modular (CM) varieties, and it has
been thought that a similar characterization of dualizability for algebras in a
CM variety would also hold. We prove that if $\A$ omits tame congruence types
$\1$ and $\5$ (all locally finite CM varieties omit these types) and is
dualizable, then $\A$ has a cube term.
\end{abstract} 

\maketitle 

\section{Introduction}  
In a variety $\V$ with some term $t(x_1,\ldots,x_n)$, the term $t$ is said to be
a \emph{cube term for $\V$} if for every $1\leq i\leq n$ there is a choice of
$u_1,\ldots,u_n\in \{x,y\}$ with $u_i = y$ such that the identity $t(u_1,\ldots,
u_n) \approx x$ holds in $\V$. Examples of cube terms include Maltsev terms and
near unanimity terms. For a variety, the property of having a cube term has been
characterized and studied in the context of the algebraic approach to the
Constraint Satisfaction Problem (for instance, see \cite{BIMMVW_FewSubalgs}) as
well as in more classic Universal Algebraic settings (for instance, see
\cite{KS_ClonesAlgParallelogram}).

The near unanimity term condition has a long-standing and particularly nice
connection with the theory of Natural Dualities. One of the early results of the
theory is that if a finite algebra has a near unanimity term, then it admits a
natural duality. Davey, Heindorf, and McKenzie~\cite{DHM_NUObstacle} prove that
a converse to this result holds if we assume that the finite algebra belongs to
a congruence distributive variety: the finite algebra $\A$ has a near unanimity
term if and only if $\V(\A)$ is congruence distributive and $\A$ admits a
natural duality.

It is well-known that the presence of a near unanimity term for a variety
implies congruence distributivity. In a similar way, the presence of a cube term
for a variety implies congruence modularity (see \cite{BIMMVW_FewSubalgs}).
Since the cube term is a generalization of the near unanimity term, it was
thought that there might be a similar connection between dualizability for
finite algebras generating congruence modular varieties and the presence of a
cube term. A stronger condition than just the presence of a cube term is
required to prove dualizability, however, since the group $\m{S}_3$ generates a
congruence modular variety with a cube term and is dualizable, but the algebra
obtained from $\m{S}_3$ by adding constant operations for every element of $S_3$
also generates a congruence modular variety with a cube term but is
non-dualizable (this example is due to Idziak).

In this paper we prove that if a finite algebra omits tame congruence types $\1$
and $\5$ and does not have a cube term, then it is inherently non-dualizable.
Algebras omitting types $\1$ and $\5$ are easy to find: every finite algebra in
a congruence modular variety omits these types. We begin with a discussion of
Natural Dualities in Section~\ref{sec:natural_dualities}, then in
Section~\ref{sec:tools} we state and give references for the tools and
techniques used in the proof, and we finish by proving the main result in
Section~\ref{sec:theorem}.
\section{Natural Dualities}\label{sec:natural_dualities}   
The primary reference for the theory of Natural Dualities is Clark and Davey
\cite{CD_NatDualitiesWorkingAlg}, and we cannot possibly hope to go into an
equivalent level of detail here. The main tool used in this paper will be a
theorem about non-dualizability stated at the start of Section~\ref{sec:tools}.
The aim of this section is to provide a definition and to give some examples of
dualizable algebras.

Let $\A = \left< A; F \right>$ be a finite algebra. The theory of Natural
Dualities aims to characterize when there is a class of structured
topological spaces $\cl{X}$ that is dually equivalent to the quasivariety
generated by $\A$.

A \emph{structured topological space} is a structure $\sts{B} = \left<
B;G,H,R,\mathcal{T} \right>$, where $G$ is a set of total operations, $H$ is
a set of partial operations, $R$ is a set of relations, and $\mathcal{T}$ is
a topology (all on $B$). The structured topological space $\sts{A} = \left<
A; G,H,R,\mathcal{T}\right>$ with the same underlying set as $\A$ is called
an \emph{alter ego of $\A$} if the topology $\mathcal{T}$ is discrete and
\[
  \big\{ \text{graph}(f) \mid f\in G\cup H \big\} \cup R 
  \subseteq \bigcup_{n\in \ZZ_{>0}} \S(\A^n),
\]
where $\text{graph}(f) = \{(x,f(x)) \mid x\in \text{dom}(f)\}$ and
$\S(\A^n)$ is the set of all subalgebras of $\A^n$. This is equivalent to
the condition that every operation in $G\cup H$ has domain equal to a
subalgebra of a power of $\A$ and is a homomorphism from that subalgebra to
$\A$ and that every relation in $R$ is a subalgebra of a power of $\A$. Fix
a particular alter ego $\sts{A}$ of $\A$. The two categories that we will be
considering are $\cl{A} = \S\P(\A)$ (the quasivariety generated by $\A$) and
$\cl{X} = \S_c\P^+(\sts{A})$ (the class of closed substructures of non-zero
powers of $\sts{A}$).

For $\m{B}\in \cl{A}$, we define the \emph{dual of $\m{B}$} to be
$\m{B}^{\partial} = \Hom(\m{B}, \A)\in \mathcal{X}$. For $\sts{B}\in \cl{X}$,
we define the \emph{dual of $\sts{B}$} to be $\sts{B}^{\partial} =
\Hom(\sts{B}, \sts{A})\in \mathcal{A}$ (the set of all continuous structure
preserving homomorphisms from $\sts{B}$ to $\sts{A}$). That $\m{B}^\partial$
and $\sts{B}^\partial$ are members of their respective categories is a
consequence of $\sts{A}$ being an alter ego of $\A$. For each $\m{B}\in
\mathcal{A}$ we have the natural mapping of ``evaluation at $x$'',
\[ \begin{array}{rlcl}
  e_{\m{B}}: & \m{B} & \to & \m{B}^{\partial\partial} \\
  & x & \mapsto & 
    \left( e_{\m{B}}(x):\m{B}^{\partial} \to \sts{A}: y \mapsto y(x)
    \right),
\end{array} \]
and it is straightforward to show that this map is injective. When for each
$\m{B}$ the mapping $e_{\m{B}}$ is additionally a surjection, then we say
that \emph{$\sts{A}$ dualizes $\m{A}$} or (when we do not wish to mention
$\sts{A}$) that \emph{$\m{A}$ admits a (natural) duality} or is
\emph{dualizable}.

Examples of algebras which admit a natural duality include 
\begin{itemize}
  \item groups whose Sylow subgroups are abelian (this is an
    equivalence)~\cite{N_NatDualGroups},
  \item commutative rings whose Jacobson radical squares to $(0)$ (this is also
    an equivalence)~\cite{CISSW_NatDualCommRings},
  \item algebras with a compatible semilattice
    operation~\cite{DJPT_NatDualSemilat}, and
  \item algebras that have a near unanimity term
    operation~\cite{DHM_NUObstacle}.
\end{itemize}
One of the main goals of the theory is to give algebraic characterizations of
dualizability instead of category-theoretical ones. Quite a lot has been
achieved to this end, for instance the characterization of dualizability in
terms of a certain kind of entailment of relations given by
Zadori~\cite{Z_DualityFinRels} and more generally by Davey, Haviar, and
Priestley~\cite{DHP_SyntaxSemanticsEntailmentDuality}.
\section{Tools}\label{sec:tools}   
The proof of the theorem contained in the next section uses several tools and
techniques from the theory of Natural Dualities and Tame Congruence Theory, as
well as some techniques used by Markovic, Maroti, and McKenzie in
\cite{MMM_FinClonesCube} that are associated with characterizing when a finite
idempotent algebra has a cube term. In this section we will state and provide
references for these.

Let $\A$ be a finite algebra. $\A$ is said to be \emph{inherently
non-dualizable} if for all finite algebras $\m{B}$ we have that $\A\in
\S\P(\m{B})$ implies $\m{B}$ is non-dualizable. Davey, Idziak, Lampe, and
McNulty~\cite{DILM_DualizabilityGraphAlgebras} give sufficient conditions for an
algebra to be inherently non-dualizable in the theorem below, and the majority
of our efforts in the next section will be to verify that the two numbered
hypotheses of this theorem hold.
\begin{thm}[\cite{DILM_DualizabilityGraphAlgebras}, Theorem $3$] 
\label{thm:inherent_non-dualizability}
Let $Z$ be an index set, $\A$ a finite algebra, $\m{B}\leq \A^Z$, and
$B_0\subseteq B$ be an infinite subset such that
\begin{enumerate}
  \item there is a function $\phi:\omega\to \omega$ such that for all $k\in
  \omega$ and all $\theta\in \Con(\m{B})$ of index at most $k$,
  $\theta|_{B_0}$ has a unique block of size greater than $\phi(k)$; and

  \item if the element $g\in A^Z$ is defined by $g(z) = a_z(z)$ for $z\in
  Z$, where $a_z$ is an element of the unique block of $\ker(\pi_z)|_{B_0}$
  of size greater than $\phi(|A|)$, then $g\not\in B$.
\end{enumerate}
Then $\A$ is inherently non-dualizable.
\end{thm}  

Congruence covers in a finite algebra can be classified into five types
(enumerated as types $\1, \ldots, \5$), and the Tame Congruence Theory of Hobby
and McKenzie~\cite{HM_TCT} gives great insight into how the presence or absence
of these types in the congruence lattices of algebras in a locally finite
variety can be recognized in terms of Maltsev conditions and congruence
conditions. Theorem 9.8 of Hobby and McKenzie~\cite{HM_TCT} proves that a
locally finite variety omits types $\1$ and $\5$ if and only if it satisfies
some idempotent Maltsev condition not satisfied by the variety of all
semilattices. Theorem 5.28 of Kearnes and
Kiss~\cite{KK_ShapeOfCongruenceLattices} proves that this
latter condition is equivalent to a single particular Maltsev condition (item
(2) of the theorem below). Putting these two results together gives us the next
Theorem.
\begin{thm}[\cite{HM_TCT}, Theorem 9.8, and   
  \cite{KK_ShapeOfCongruenceLattices}, Theorem 5.28] 
\label{thm:term_cond_omit_1_5}
The following are equivalent for a locally finite variety $\V$.
\begin{enumerate}
  \item $\V$ omits types $\1$ and $\5$.

  \item $\V$ has a sequence of idempotent terms $f_i(x,y,u,v)$ for $0\leq
  i\leq 2m+1$, such that 
  \begin{enumerate}
    \item $\V\models f_0(x,y,u,v) \approx x$ and $\V\models
      f_{2m+1}(x,y,u,v) \approx v$,
    \item $\V\models f_i(x,y,y,y) \approx f_{i+1}(x,y,y,y)$ for all even $i$,
    \item $\V\models f_i(x,x,y,y) \approx f_{i+1}(x,x,y,y)$ for all odd $i$,
      and
    \item $\V\models f_i(x,y,x,y) \approx f_{i+1}(x,y,x,y)$ for all odd $i$.
  \end{enumerate}
\end{enumerate}
\end{thm}  
If a locally finite variety is congruence modular then it omits types $\1$ and
$\5$. Thus, finite algebras in a congruence modular variety have terms
satisfying the Maltsev condition of the above theorem. In fact, by reindexing
and rearranging some of the variables, the Day Terms introduced in
\cite{D_DayTerms} satisfy this Maltsev condition.

Markovic, Maroti, and McKenzie~\cite{MMM_FinClonesCube} provide a useful
characterization of those finite idempotent algebras that have cube terms, which
we will now summarize. Fix an algebra $\A$ and elements $a,b\in A$. If there is
a term $t(x_1,\ldots,x_n)$ and tuples $u_i\in \{a,b\}^m\setminus \{a\}^m$ for
some $m\in \ZZ_{>0}$ such that
\[
  t(u_1(j),\ldots,u_n(j)) = a
\]
for all $1\leq j\leq m$, then we will write $a\prec b$. Observe that if $\A$
has a cube term then $a\prec b$ for all $a,b\in A$. If $\A$ has subalgebras
$\m{D}\lneq \m{B}\leq \A$ such that for every term $t(x_1,\ldots,x_n)$ there
is some $i$ with
\[
  t(B,\ldots,\stackrel{i}{\hat{D}},\ldots,B) \subseteq D,
\]
then the pair $(D,B)$ is called a \emph{cube term blocker} for $\A$. 

\begin{thm}[\cite{MMM_FinClonesCube}, Theorem 2.1]  
\label{thm:cube_blockers}
Let $\A$ be a finite idempotent algebra. Then $\A$ has a cube term if and
only if it has no cube term blockers.
\end{thm}  

Suppose that $\A$ is finite, idempotent, and does not have a cube term, and let
$\m{B}\leq \A$ be minimal for not having a cube term. In this case, the cube
term blocker for $\A$ can be taken to be of the form $(D,B)$, and we can make
two useful observations about $\m{B}$ and $\m{D}$:
\begin{itemize}
  \item if $u,v\in B$ are such that $u\not\prec v$, then $\{u,v\}$ generates
    $\m{B}$;
  \item if $v\in D$ and $u\in B\setminus D$, then $u\not\prec v$ and thus
    $\{u,v\}$ generates $\m{B}$.
\end{itemize}
Both observations follow from $\m{B}$ being minimal for not having a cube term
and from $(D,B)$ being a cube term blocker. These observations and
Theorem~\ref{thm:cube_blockers} will be the starting point for the proof of the
theorem contained in the next section.

The last tool that we will need is the existence of a weak near unanimity
term. A term $t(x_1,\ldots,x_n)$ of $\V$ is said to be a \emph{weak near
unanimity term for $\V$} if it is idempotent and
\[
  \V\models 
  t(y,x,\ldots,x) 
  \approx t(x,y,x,\ldots,x) 
  \approx \cdots 
  \approx t(x,\ldots,x,y).
\]
For finitely generated idempotent varieties $\V$, Maroti and McKenzie
\cite{MM_WNU} show that $\V$ has a weak near unanimity term of arity at least
$2$ if and only if $\V$ omits type $\1$ (such varieties are called \emph{Taylor}
varieties).
\section{The Theorem}\label{sec:theorem}   
\begin{thm}\label{thm:omits_1_5_no_cube_then_non-dualizable}   
Let $\A$ be a finite algebra such that $\V(\A)$ omits types $\1$ and $\5$.
If $\A$ does not have a cube term, then $\A$ is inherently non-dualizable.
\end{thm}
\begin{proof}
Assume that $\A$ does not have a cube term, and let $\A_I$ be the idempotent
reduct of $\A$. Observe that $\A_I$ also does not have a cube term, and select
idempotent $\m{B}\leq \A_I$ minimal such that $\m{B}$ does not have a cube term.
Since $\m{B}$ is idempotent and minimal for not having a cube term, by the
observations in Section~\ref{sec:tools} we can fix a cube term blocker $(D,B)$
for $\m{B}$ and elements $a\in B\setminus D$ and $b\in D$ such that $a\not\prec
b$ in $\A_I$. 

Enumerate the elements of $\A$ as $A = \{ a_0, a_{-1}, \ldots, a_{-n} \}$,
and define elements of $A^\ZZ$
\[
  \alpha_{i_1 i_2 \cdots i_n}^{y_1 y_2 \cdots y_n}(j) = \begin{cases}
    a_j & \text{if } j\in [-n,0], \\
    y_k & \text{if } j = i_k, \\
    a   & \text{otherwise}, \end{cases}
\]
for any $y_1,\ldots, y_n\in A$ and $i_1,\ldots,i_n\not\in [-n,0]$. If all
of the $y_i$ are equal to $b$, then we will omit them from the notation.
That is,
\[
  \alpha_{i_1 \cdots i_n}(j) = \begin{cases} 
    a_j & \text{if } j\in [-n,0], \\
    b & \text{if } j\in \{i_1,\ldots,i_n\}, \\
    a & \text{otherwise}, \end{cases}
\]
for $i_1,\ldots, i_k\not\in [-n,0]$. Let
\[
  C_0 
  = \left\{ \alpha_i \mid i\in \ZZ\setminus [-n,0] \right\}
  \qquad \qquad \text{and} \qquad \qquad
  \m{C} 
  = \Sg^{\A^\ZZ}( C_0 )
\]
(note that $\m{C}$ need not be idempotent). When we are performing calculations
in $\m{C}$ using idempotent terms, we will omit calculations like $t(a,\ldots,
a) = a$ and $t(a_j,\ldots, a_j) = a_j$ for $j\in [-n,0]$.

We will apply Theorem~\ref{thm:inherent_non-dualizability} to this situation
to show that $\A$ is inherently non-dualizable. Our first step is to show
that (1) of that theorem holds. Let $\phi:\ZZ\to\ZZ$ be defined to be the
constant function $\phi(k) = 1$, and suppose that $\theta\in \Con(\m{C})$
has finite index and that $\theta|_{C_0}$ has two blocks 
\[
  \left\{ \alpha_i \mid i\in S \right\}
  \qquad\qquad \text{and} \qquad\qquad
  \left\{ \alpha_i \mid i\in T \right\},
\]
with $|S|, |T| > 1$ and $S\cap T = \emptyset$. For ease of writing, say
$1,3\in S$ and $2,4\in T$

\begin{claim}
The set
\[
  \big\{ \alpha_{mn} \mid (m,n)\in \{1,3\}\times \{2,4\} \big\} \cup 
    \big\{ \alpha_{mnk} \mid 
    (m,n,k)\in \{1,3\} \times \{2,4\} \times \{1,2,3,4\} \big\}
\]
is contained in a single $\theta$-block.
\end{claim}
\begin{claimproof}
We will frequently use the fact that if $u\in D$, then by the minimality of
$\m{B}$ the set $\{u,a\}$ generates $\m{B}$ via idempotent terms of $\A_I$.
$\V(\A)$ omits type $\1$, so let $w(x_1,\ldots,x_n)$ be a weak near
unanimity term for $\A$ (and hence for $\A_I$). If $u\in D$ and $v\in B$,
then since $(D,B)$ is a cube term blocker for $\m{B}$,
\begin{align*}
  & w(u,v,\ldots,v)
    = w(v,u,v,\ldots,v)
    = \cdots
    = w(v,v,\ldots,v,u)\in D & \text{and} \\
  & w(v,u,\ldots,u)
    = w(u,v,u,\ldots,u)
    = \cdots
    = w(u,u,\ldots,u,v)\in D.
\end{align*}
From this and since $\alpha_1 \;\theta\; \alpha_3$ and $\alpha_2 \;\theta\;
\alpha_4$, it follows that
\[
  w(\alpha_1,\alpha_2,\ldots,\alpha_2)
  = w\pmat{ \vdots &   &        & \vdots \\
            b      & a & \cdots & a      \\
            a      & b & \cdots & b      \\
            \vdots &   &        & \vdots }
  = \alpha_{12}^{cd}
  \qquad \text{and} \qquad
  \alpha_{12}^{cd} 
  \;\theta\; \alpha_{32}^{cd} 
  \;\theta\; \alpha_{14}^{cd} 
  \;\theta\; \alpha_{34}^{cd},
\]
for some $c,d\in D$. By the minimality of $\m{B}$, the set $\{c,a\}$ must
idempotently generate $\m{B}$. Thus there is an idempotent term $t(x,y)$
such that $b = t(c,a)$. Therefore
\[
  t(\alpha_{12}^{cd}, \alpha_2)
  = t\pmat{\vdots & \vdots \\
           c      & a      \\
           d      & b      \\
           \vdots & \vdots }
  = \alpha_{12}^{be}
  \qquad\qquad \text{and} \qquad\qquad
  \alpha_{12}^{be} 
  \;\theta\; \alpha_{32}^{be} 
  \;\theta\; \alpha_{14}^{be} 
  \;\theta\; \alpha_{34}^{be},
\]
for some $e\in D$ since $D$ is a subuniverse of $\A_I$ and $d,b\in D$. Since
$e\in D$, the set $\{e,a\}$ idempotently generates $\m{B}$. Thus there is an
idempotent term $s(x,y)$ such that $b = s(e,a)$. Therefore
\[
  s(\alpha_{12}^{be}, \alpha_1)
  = s\pmat{\vdots & \vdots \\
           b      & b      \\
           e      & a      \\
           \vdots & \vdots }
  = \alpha_{12}.
  \qquad\qquad \text{and} \qquad\qquad
  \alpha_{12} 
  \;\theta\; \alpha_{32}
  \;\theta\; \alpha_{14}
  \;\theta\; \alpha_{34}.
\]
Using the weak near unanimity term again,
\[
  w(\alpha_{12}, \alpha_3,\ldots, \alpha_3)
  = w\pmat{ \vdots &   &        & \vdots \\
            b      & a & \cdots & a      \\
            b      & a & \cdots & a      \\
            a      & b & \cdots & b      \\
            \vdots &   &        & \vdots }
  = \alpha_{123}^{ccd}
\]
and
\[
  \alpha_{123}^{ccd} 
  \;\theta\; \alpha_{143}^{ccd}
  \;\theta\; \alpha_{12}^{bc} 
  \;\theta\; \alpha_{32}^{bc} 
  \;\theta\; \alpha_{14}^{bc}
  \;\theta\; \alpha_{34}^{bc},
\]
Using the term $t(x,y)$ again, we have
\[
  t(\alpha_{123}^{ccd}, \alpha_3)
  = t\pmat{ \vdots & \vdots \\
            c      & a      \\
            c      & a      \\
            d      & b      \\
            \vdots & \vdots }
  = \alpha_{123}^{bbe}
  \qquad \text{and} \qquad
  \alpha_{123}^{bbe} 
  \;\theta\; \alpha_{143}^{bbe} 
  \;\theta\; \alpha_{12}
  \;\theta\; \alpha_{32}
  \;\theta\; \alpha_{14}
  \;\theta\; \alpha_{34}.
\]
Using the term $s(x,y)$ again, we have
\[
  s(\alpha_{123}^{bbe}, \alpha_{12})
  = s\pmat{ \vdots & \vdots \\
            b      & b      \\
            b      & b      \\
            e      & a      \\
            \vdots & \vdots }
  = \alpha_{123}
  \qquad \text{and} \qquad
  \alpha_{123} 
  \;\theta\; \alpha_{134}
  \;\theta\; \alpha_{12}
  \;\theta\; \alpha_{32}
  \;\theta\; \alpha_{14}
  \;\theta\; \alpha_{34}.
\]
A similar argument will give us that $\alpha_{124} \;\theta\; \alpha_{234}
\;\theta\; \alpha_{12}$ as well, completing the proof of the claim.
\end{claimproof}

\begin{claim}
$\alpha_1 \;\theta\; \alpha_2$.
\end{claim}
\begin{claimproof}
$\V(\A)$ omits types $\1$ and $\5$, so let $f_i(x,y,u,v)$ for $0\leq i\leq
2m+1$ be the idempotent terms from Theorem~\ref{thm:term_cond_omit_1_5}. If
$i$ is even then
\[
  f_i(\alpha_1,\alpha_{12},\alpha_{12},\alpha_{12}) 
  = f_{i+1}(\alpha_1,\alpha_{12},\alpha_{12},\alpha_{12}).
\]
If $i$ is odd then by the previous claim,
\begin{align*}
  f_i(\alpha_1, \alpha_{12}, \alpha_{12}, \alpha_{12})
  & \;\theta\; f_i(\alpha_1, \alpha_{12}, \alpha_{34}, \alpha_{234}) \\
  & = f_i\pmat{\vdots &   &   & \vdots \\
               b      & b & a & a      \\
               a      & b & a & b      \\
               a      & a & b & b      \\
               a      & a & b & b      \\
               \vdots &   &   & \vdots }
  = f_{i+1}\pmat{\vdots &   &   & \vdots \\
                 b      & b & a & a      \\
                 a      & b & a & b      \\
                 a      & a & b & b      \\
                 a      & a & b & b      \\
                 \vdots &   &   & \vdots }  \\
  & = f_{i+1}(\alpha_1, \alpha_{12}, \alpha_{34}, \alpha_{234})
  \;\theta\; f_{i+1}(\alpha_1, \alpha_{12}, \alpha_{12}, \alpha_{12}).
\end{align*}
Combining both of these, we have that $f_i(\alpha_1, \alpha_{12},
\alpha_{12}, \alpha_{12}) \;\theta\; f_{i+1}(\alpha_1, \alpha_{12},
\alpha_{12}, \alpha_{12})$ for all $i$, so
\[
  \alpha_1
  = f_0(\alpha_1,\alpha_{12},\alpha_{12},\alpha_{12})
  \;\theta\; f_{2m+1}(\alpha_1,\alpha_{12},\alpha_{12},\alpha_{12})
  = \alpha_{12}.
\]
A similar argument will show that $\alpha_{2} \;\theta\; \alpha_{12}$ as
well.
\end{claimproof}

Returning to the main proof, we now have $\alpha_1\;\theta\; \alpha_2$,
which contradicts $S\cap T = \emptyset$. Therefore there can be only one
block of $\theta|_{C_0}$ of size greater than $1$. This is item (1) from
Theorem~\ref{thm:inherent_non-dualizability}. 

We will now prove that item (2) from
Theorem~\ref{thm:inherent_non-dualizability} also holds. Let $\alpha\in
A^\ZZ$ be defined by
\[
  \alpha(j) = \begin{cases}
    a_j & \text{if } j\in [-n,0], \\
    a   & \text{otherwise}
  \end{cases}
\]
(recall that elements of $\A$ were enumerated $A =
\{a_0,a_{-1},\ldots,a_{-n}\}$).  Let $g\in A^\ZZ$ be the element defined in
item (2) of Theorem~\ref{thm:inherent_non-dualizability}. That is, $g(j) =
\pi_j(c_j)$, where $c_j$ is a member of the unique non-singleton block of
$\ker(\pi_j)|_{C_0}$.

\begin{claim}
$g = \alpha$ and $\alpha\not\in C$.
\end{claim}
\begin{claimproof}
We first show that $g = \alpha$. If $j\in [-n,0]$, then $\ker(\pi_j)|_{C_0}$
consists of a single block, and $g(j) = a_j = \alpha(j)$. If $j\not\in [-n,0]$,
then $\ker(\pi_j)|_{C_0}$ consists of two blocks,
\[
  X_j = \{ \alpha_j \} 
  \qquad\qquad \text{and} \qquad\qquad
  Y_j = \{ \alpha_i \mid i\neq j \}
\]
($\pi_j(X_j) = b$ and $\pi_j(Y_j) = a$), and $g(j) = a = \alpha(j)$. Therefore
$g = \alpha$.

We now show that $\alpha\not\in C$. Suppose to the contrary that $\alpha\in
C$. Then there exists a term $t(x_1,\ldots,x_m)$ of $\A$ such that
$t(\alpha_1, \ldots, \alpha_m) = \alpha$ for some $m$. That is,
\[
  t\pmat{ \vdots &        &        & \vdots \\
          a_{-n} & a_{-n} & \cdots & a_{-n} \\
          \vdots &        &        & \vdots \\
          a_0    & a_0    & \cdots & a_0    \\
          b      & a      & \cdots & a      \\
          a      & b      &        & a      \\
          \vdots &        & \ddots & \vdots \\
          a      & a      & \cdots & b      \\
          \vdots &        &        & \vdots }
  = \pmat{ \vdots \\
           a_{-n} \\
           \vdots \\
           a_0    \\
           a      \\
           a      \\
           \vdots \\
           a      \\
           \vdots }.
\]
Since $A = \{a_0,\ldots, a_{-n}\}$, the ``top'' portion of the equality
implies that $t(x_1,\ldots,x_m)$ is idempotent and hence is a term of
$\A_I$. The ``bottom'' portion of the equality then contradicts $a\not\prec
b$ in $\A_I$.
\end{claimproof}

This completes the proof that item (2) from
Theorem~\ref{thm:inherent_non-dualizability} holds. Thus $\A$ is inherently
non-dualizable.
\end{proof}  

\begin{cor}\label{cor:omits_1_5_dualizable_then_cube}   
Let $\A$ be a finite algebra such that $\V(\A)$ omits types $\1$ and $\5$.
If $\A$ admits a natural duality, then $\A$ has a cube term.
\end{cor}  

\bibliographystyle{amsplain}
\bibliography{dual-implies-cube}
\begin{center}
  \rule{0.61803\textwidth}{0.1ex}   
\end{center}
\end{document}